\documentclass[12pt]{article}
\usepackage{amsfonts}
\usepackage{amsmath,amscd, amsthm, tabularx}
\newtheorem{prop}{Proposition}[section]
\newtheorem{thm}[prop]{Theorem}

\newtheorem{lemma}[prop]{Lemma}

\theoremstyle{remark}
\newtheorem{rmk}[prop]{Remark}
\usepackage{graphicx}
\usepackage{cite}
\usepackage{url}

\title{Extremal Transitions from Nested Reflexive Polytopes}
\author{Karl Fredrickson}
\date{}
\begin{document}
\maketitle
\begin{abstract}In this paper we look at the question of when an inclusion of reflexive polytopes determines a torically-defined extremal transition between smooth Calabi-Yau hypersurface families.  We show this is always possible in dimensions two and three.  However, in dimension four and higher, obstructions can occur.  This leads to a smooth projective family of Calabi-Yau threefolds that is birational to one of Batyrev's hypersurface families, but topologically distinct from all such families.\end{abstract}
\section{Introduction}
A well-known idea in the study of Calabi-Yau varieties is the use of nested reflexive polytopes to connect the moduli spaces of two Calabi-Yau families.  By Batyrev's construction in \cite{Ba}, any reflexive polytope $\Delta_1$ determines a family of Calabi-Yau hypersurfaces in a resolution of the associated toric variety $X(\Delta_1)$, obtained from the fan of cones over proper faces of $\Delta_1$.  These hypersurfaces are defined by generic sections of the anticanonical bundle of the ambient toric variety, which has a basis of monomial global sections in bijection with lattice points in the dual polytope $\Delta^*_1$. Suppose that $\Delta_1$ is contained in a larger reflexive polytope $\Delta_2$.  Then we have a reverse inclusion $\Delta^*_2 \subseteq \Delta^*_1$.  By allowing coefficients on the monomials associated to lattice points in $\Delta^*_1 \backslash \Delta^*_2$ to approach zero, we obtain a degeneration of smooth Calabi-Yau hypersurfaces in a resolution of $X(\Delta_1)$.  The singular fibers will be birational to Calabi-Yau varieties in $X(\Delta_2)$.  

This strategy was used in \cite{web} and \cite{acjm} to show that all families of smooth Calabi-Yau threefolds in (resolutions of) four-dimensional weighted projective spaces can be connected by extremal transitions.  An extremal or geometric transition is a transition from one nonsingular Calabi-Yau variety $X$ to another Calabi-Yau variety $Y$, in which $X$ degenerates to a singular variety $X_0$, and $Y$ is obtained as a resolution of singularities of $X_0$.  This includes the more familiar case of conifold transitions, where $X_0$ has isolated ordinary double points as singularities and $Y$ is a small resolution of $X_0$.  For more on the subject of extremal transitions, see the papers \cite{morrison} and \cite{rossi}.

After Batyrev introduced his reflexive polytope construction, Kreuzer and Skarke classified all possible reflexive polytopes in four dimensions up to lattice equivalence (473,800,776, as it turns out!) in \cite{ks}.  As a byproduct of their classification, they were able to show that any pair of these reflexive polytopes can be connected by a chain of inclusions, thus showing that the moduli spaces are also connected in this class of Calabi-Yau threefolds.

The main purpose of this paper is to examine when an inclusion of reflexive polytopes $\Delta_1 \subseteq \Delta_2$ leads to a torically defined geometric transition between associated Calabi-Yau families $\mathcal{X}(\Delta_1)$ and $\mathcal{X}(\Delta_2)$.  The toric varieties $X(\Delta_i)$ are defined by fans $\Sigma(\Delta_i)$ which are given by taking cones over the proper faces of $\Delta_i$, and the families $\mathcal{X}(\Delta_i)$ exist in MPCP (maximal projective crepant partial, as defined in \cite{Ba}) resolutions $\widehat{X}(\Delta_i)$ of $X(\Delta_i)$ associated to fans $\widehat{\Sigma}(\Delta_i)$ which are subdivisions of the original fans $\Sigma(\Delta_i)$.  

By letting coefficients on the appropriate monomials go to zero, it is always possible to degenerate the family $\mathcal{X}(\Delta_1)$ to a singular family $\mathcal{X}_0(\Delta_1)$ which is birational to $\mathcal{X}(\Delta_2)$.  However, the question of whether the singular family can be torically resolved to $\mathcal{X}(\Delta_2)$ depends on whether there exists a toric morphism $\widehat{X}(\Delta_2) \rightarrow \widehat{X}(\Delta_1)$.  This is the same as asking whether MPCP subdivisions of $\Sigma(\Delta_i)$ exist such that $\widehat{\Sigma}(\Delta_2)$ is a subdivision of $\widehat{\Sigma}(\Delta_1)$.  This is a nontrivial convex geometry problem that does not always have a solution. 

In the two-dimensional case, only one MPCP subdivision of a given reflexive polytope is possible, and it is relatively easy to show that the MPCP resolutions of nested two-dimensional reflexive polytopes $\Delta_1 \subseteq \Delta_2$ will always be compatible with each other.  In three dimensions, we show that any MPCP subdivision $\widehat{\Sigma}(\Delta_1)$ can be refined to an MPCP subdivision $\widehat{\Sigma}(\Delta_2)$.  In four dimensions there are cases where $X(\Delta_1)$ and $X(\Delta_2)$ are already smooth (so smooth hypersurface families exist in both $X(\Delta_i)$ and no resolution is needed) but the toric morphism $X(\Delta_2) \rightarrow X(\Delta_1)$ fails to exist.  We will see that this counterexample leads to a smooth projective Calabi-Yau family which is birational to one of Batyrev's hypersurface families, but topologically distinct from all such families.

\subsection{Basic definitions and notation}
The notion of a reflexive polytope was originally introduced by Batyrev in his paper \cite{Ba}.  Suppose we have a lattice $N \cong \mathbb{Z}^n$ for some $n$.  We can consider the corresponding real vector space $N_\mathbb{R} = N \otimes \mathbb{R}$, the dual lattice $M = Hom(N, \mathbb{Z})$, and the dual vector space $M_\mathbb{R} = M \otimes \mathbb{R}$.  A reflexive polytope $\Delta \subseteq N_\mathbb{R}$ is a lattice polytope with the origin in its interior, such that the dual polytope $\Delta^* \subseteq M_\mathbb{R}$ is also a lattice polytope.  The dual polytope is defined as 
$$\Delta^* = \{ m \in M_\mathbb{R} \ | \ \langle m, n \rangle \geq -1, \forall n \in \Delta \},$$ 
where $\langle , \rangle$ is the natural real-valued pairing between $M_\mathbb{R}$ and $N_\mathbb{R}$.  Any reflexive polytope will have the origin as its only interior lattice point, but in dimensions three and higher there are examples of lattice polytopes with this property that are not reflexive.

By ``cone over'' a subset $S$ of a real vector space, we will always mean the set \[ \mathbb{R}_{\geq 0} S = \{ r s \ | \ r \in \mathbb{R}, r \geq 0, s \in S \}. \]

Given a reflexive polytope $\Delta$ (or any lattice polytope with the origin in its interior), we can consider the toric variety $X(\Delta)$ associated to $\Delta$.  This is defined as the toric variety associated to the fan $\Sigma(\Delta)$ which has cones consisting of cones over proper faces of $\Delta$.  As long as $\Delta$ is reflexive, $X(\Delta)$ will have at most Gorenstein singularities, and the Newton polytope of global sections of its anticanonical bundle is just the dual polytope $\Delta^*$.  

If $\Sigma$ is a complete fan, then any integral piecewise linear function $\varphi : \Sigma \rightarrow \mathbb{R}$ determines a line bundle $\mathcal{L}$ on $X(\Sigma)$.  As a vector space over $\mathbb{C}$, global sections of $\mathcal{L}$ have a basis consisting of lattice points in the Newton polytope of $\varphi$, defined as 
$$Newt(\varphi) = \{ m \in M_\mathbb{R} \ | \ \langle m, n \rangle \geq -\varphi(n), \forall n \in N_\mathbb{R} \}.$$
Note that the convention for $\varphi$ we use is the opposite of some sources such as \cite{CLS}.  They would use $-\varphi$ to obtain the same bundle $\mathcal{L}$.  Thus, for instance, we associate ample line bundles with strictly lower convex functions $\varphi$ on $\Sigma$, where lower convex means that for any $a, b \in [0,1]$ with $a+b = 1$ and $n_1, n_2 \in N_\mathbb{R}$, 
$$\varphi(a n_1 + b n_2) \leq a \varphi (n_1) + b \varphi(n_2).$$

Let $\Delta \subseteq N_\mathbb{R}$ be a reflexive polytope.  A subdivision $\widehat{\Sigma}(\Delta)$ of the fan $\Sigma(\Delta)$ is called an MPCP subdivision if it is projective, the maximal cones of $\widehat{\Sigma}(\Delta)$ are cones over elementary simplices (simplices containing no lattice points except their vertices), and the primitive integral generators of all rays in $\widehat{\Sigma}(\Delta)$ are lattice points in the boundary of $\Delta$.  The main purpose of MPCP subdivisions in \cite{Ba} is to resolve families of Calabi-Yau hypersurfaces.  If $\Delta$ is dimension four or less, then an MPCP resolution of $X(\Delta)$ will resolve all generic members of the family $\mathcal{X}(\Delta) \subseteq X(\Delta)$ to smooth Calabi-Yau hypersurfaces.  

MPCP subdivisions of reflexive polytopes always exist.  However, they are in general not unique, unless $\Delta$ is two-dimensional.

\subsection{Two dimensional case}

If $\Delta \subseteq \mathbb{R}^2$ is a two-dimensional reflexive polytope, then there is exactly one MPCP subdivision $\widehat{\Sigma}(\Delta)$ of the fan $\Sigma(\Delta)$.  This is the fan whose maximal cones are cones over line segments $L = Conv(v_1, v_2)$, where $v_1$ and $v_2$ are adjacent lattice points in an edge of $\Delta$.

\begin{thm} \label{twodim} Suppose $\Delta_1 \subseteq \Delta_2$ are two-dimensional reflexive polytopes.  Then if $\widehat{\Sigma}(\Delta_i)$ is the unique MPCP subdivision of $\Sigma(\Delta_i)$, we have that $\widehat{\Sigma}(\Delta_2)$ is a refinement of $\widehat{\Sigma}(\Delta_1)$. \end{thm}

\begin{proof} If $F_1$ and $F_2$ are any complete two-dimensional fans, then $F_2$ is a refinement of $F_1$ if and only if every ray of $F_1$ is also a ray of $F_2$.  Since the rays of $\widehat{\Sigma}_i$ in this case are the rays over nonzero lattice points in $\Delta_i$, the statement follows immediately from the fact that $\Delta_1 \subseteq \Delta_2$. \end{proof}

For an illustration of a possible case, see Figure \ref{nestedsubdivs}.

\begin{figure}[t] 
  \begin{center} \includegraphics[scale=0.75]{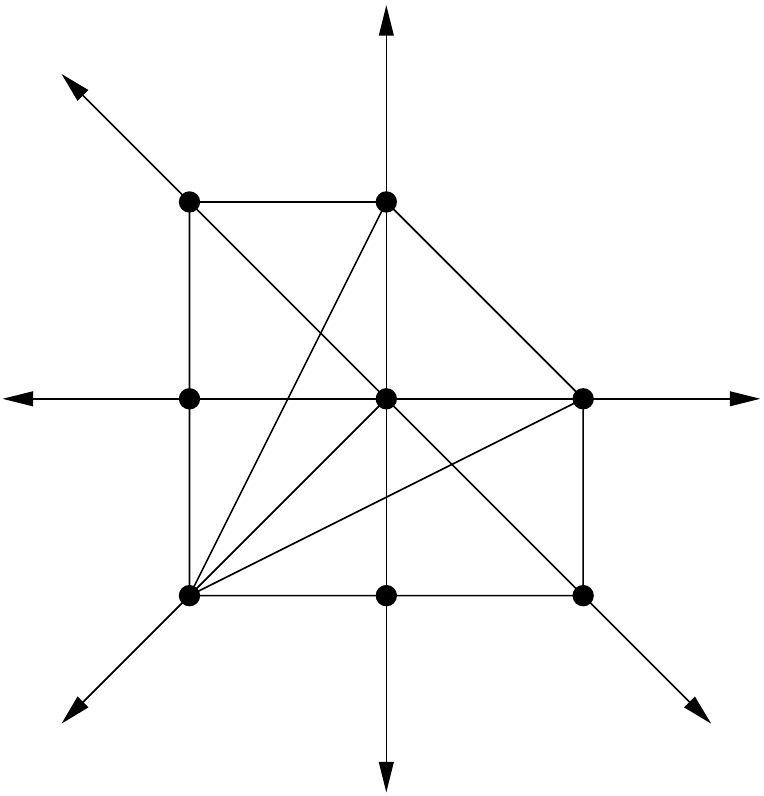} \end{center}
\caption{\label{nestedsubdivs} An inclusion of reflexive polytopes $\Delta_1 \subseteq \Delta_2$ and the corresponding MPCP subdivisions.}
\end{figure}

\section{Three dimensional case}

In three dimensions, reflexive polytopes can have many possible MPCP subdivisons.  We will prove that, given an inclusion $\Delta_1 \subseteq \Delta_2$ of three dimensional reflexive polytopes and any choice of MPCP subdivision $\widehat{\Sigma}(\Delta_1)$ of $\Sigma(\Delta_1)$, it is possible find an MPCP subdivision $\widehat{\Sigma}(\Delta_2)$ which is a refinement of $\widehat{\Sigma}(\Delta_1)$.

\begin{prop} \label{threedim} Let $\Delta$ be a three-dimensional reflexive polytope, let $C_2 \in \Sigma(\Delta)$ be the cone over an edge $E$ of $\Delta$, and let $C_1$ be a two-dimensional cone of the form $$C_1 = \mathbb{R}_{\geq 0} v_1 + \mathbb{R}_{\geq 0} v_2$$ where $v_i$ are lattice points in $\Delta$.  Then if $C_1 \cap C_2$ is one-dimensional, it must be the ray over a lattice point in $\Delta$. \end{prop}

%\begin{prop} \label{threedim} If $\Delta_1$ and $\Delta_2$ are three-dimensional reflexive polytopes with $\Delta_1 \subseteq \Delta_2$, then the primitive integral generators of rays in the intersection fan $\Sigma_{1,2} = \Sigma(\Delta_1) \cap \Sigma(\Delta_2)$ are contained in $\Delta_2$. \end{prop}

\begin{proof} %The cones in $\Sigma_{1,2}$ are in bijective correspondence with pairs of cones $C_1 \in \Sigma(\Delta_1)$ and $C_2 \in \Sigma(\Delta_2)$ such that $C_1 \cap C_2$ intersects the relative interior of both $C_1$ and $C_2$.  In this case, the dimension of $C_1 \cap C_2$ will be at least $$\hbox{dim} \ C_1 + \hbox{dim} \ C_2 -3.$$  If $C_1$ and $C_2$ are a pair corresponding to a ray of $\Sigma_{1,2}$, then only possibilities for the dimensions of $C_1$ and $C_2$ are one and one, one and two, one and three, or two and two.  In case one of $C_1$ or $C_2$ is one dimensional, then its primitive integral generator will automatically be contained in $\Delta_2$.        

The edge $E$ of $\Delta$ can be written uniquely as the intersection of two two-dimensional faces of $\Delta$, $f_1$ and $f_2$.  Let $m_1$ and $m_2$ be the vertices of $\Delta^*$ dual to $f_1$ and $f_2$, so that $\langle m_1, f_1 \rangle = -1$ and $\langle m_2, f_2 \rangle = -1$.  We can assume that $C_1 \cap C_2$ intersects the relative interior of both cones $C_i$, since otherwise the intersection would have to be a boundary ray of one of the cones, which by definition is a ray over a lattice point in $\Delta$.

{\bf Claim.}  Let $A$ be the $2 \times 2$ integer matrix defined by $a_{ij} = \langle m_i, v_j \rangle$.  Then if $C_1 \cap C_2$ is one-dimensional and intersects the relative interior of both $C_1$ and $C_2$, $A$ must have the form
\[\left(
\begin{array}{cc}
-1 & a \\
b & -1 
\end{array}\right) \]
or 
\[\left(
\begin{array}{cc}
a & -1 \\
-1 & b 
\end{array} \right)
\]
with one of $a$ or $b$ equal to zero and the other greater than or equal to zero.
\begin{proof} First we establish the existence and position of the negative ones in the matrix.  Suppose there was a row of $A$ that did not contain a $-1$.  Then for either $i=1$ or $2$, $\langle m_i, v_1 \rangle \geq 0$ and $\langle m_i, v_2 \rangle \geq 0$.  This implies that $\langle m_i, c \rangle \geq 0$ for every $c \in C_1$.  Since $m_1$ and $m_2$ must both evaluate negatively on every nonzero element of $C_2$, this means $C_1 \cap C_2$ must be the zero vector and is not one-dimensional.

Now suppose there was a column that did not contain a $-1$.  Then one of the $v_j$ is such that $\langle m_1, v_j \rangle$ and $\langle m_2, v_j \rangle$ are both at least zero.  For $C_1$ to intersect $C_2$ nontrivially, the other primitive integral generator $v_i$ must then be such that both $\langle m_1, v_i \rangle = \langle m_2, v_i \rangle = -1$.  But this implies that $v_i \in C_2$, so the intersection of the two cones must be the ray over $v_i$, which does not intersect the relative interior of both.  

To prove the statement about $a$ and $b$, suppose that at least one of $a$ or $b$ was equal to $-1$.  Then we would have a column with two negative ones, or a $v_i$ with $\langle m_1, v_i \rangle = \langle m_2, v_i \rangle = -1$, which as in the previous paragraph implies that $C_1 \cap C_2$ is the ray over $v_i$.  So we must show that not both $a$ and $b$ are positive.  Let $u$ be a nonzero vector contained in the ray $C_1 \cap C_2$.  We have $u = c_1v_1+c_2v_2$ with $c_1, c_2 > 0$ and unique.  

Suppose that $c_2 \geq c_1$ and the matrix has the first form, so that $a$ is the $(1,2)$ entry.  Then $\langle m_1, u \rangle = c_1 \langle m_1, v_1 \rangle + c_2 \langle m_1, v_2 \rangle = -c_1+ac_2$, but if $a$ is positive then $a \geq 1$ and $-c_1+ac_2 \geq 0$, contradicting that $u \in C_2$.  Therefore $c_2 \geq c_1$ implies $a=0$.  If $c_1 \geq c_2$, then a similar argument with $m_2$ in place of $m_1$ works to show that $b$ must be zero.  Thus in all cases, at least one of $a$ or $b$ must equal zero.

If the matrix has the second form, so that $a$ is the $(1,1)$ entry, it can be transformed into the first form by relabeling $v_1 \rightarrow v_2$ and $v_2 \rightarrow v_1$, and the above argument can still be used.
\end{proof}

Now we will assume (again, possibly by relabeling $v_1$ and $v_2$) that 
\[ A= \left(
\begin{array}{cc}
-1 & a \\
b & -1 
\end{array} \right).
 \]
Suppose that $a=0$.  Then the lattice point $q = v_1+(b+1)v_2$ will have $\langle m_1, q \rangle = -1$, and $\langle m_2, q \rangle = -1$.  This means that $q$ is the primitive integral generator of $C_1 \cap C_2$, and also that $q \in f_1$, because $f_1$ consists of the points in $\Delta$ on which $m_1$ and $m_2$ evaluate to $-1$.  Thus the primitive integral generator of $C_1 \cap C_2$ lies in $\Delta_2$ and we are done.  The other possible case is $b=0$.  Then we would use the lattice vector $q = (a+1)v_1+v_2$ and repeat the same argument.
\end{proof}

\begin{rmk} The reflexive assumption is essential for Proposition \ref{threedim} to be true, even if $\Delta$ is a lattice polytope with only the origin in its interior.  For a counterexample, consider the polytope $\Delta$ with nine total vertices $(\pm 1, \pm 1, -1)$, $(1,1,1)$, $(-1,1,0)$, $(-1,0,1)$, $(1,-1,1)$, and $(-1,1,1)$.  This polytope can be thought of as a cube with one corner cut off, and is non-reflexive because the face \[ Conv((1,1,1),(-1,1,0),(-1,0,1)) \] produces a non-Gorenstein singularity.  If we let $C_2$ be the cone over the edge $Conv((1,1,1),(-1,1,0))$, and $C_1$ be the cone over the line segment $Conv((0,1,0),(0,0,1))$, then $C_1$ and $C_2$ meet the conditions of Proposition \ref{threedim}, but their intersection is the ray over the lattice point $(0,2,1)$, which is not contained in $\Delta$. \end{rmk}
 
Now we can prove:

\begin{thm} \label{transexistence} Suppose that $\Delta_1 \subseteq \Delta_2$ are three-dimensional reflexive polytopes, and let $\widehat{\Sigma}(\Delta_1)$ be any MPCP subdivision of $\Sigma(\Delta_1)$.  Then there exists an MPCP subdivision $\widehat{\Sigma}(\Delta_2)$ of $\Sigma(\Delta_2)$ such that $\widehat{\Sigma}(\Delta_2)$ is a refinement of $\widehat{\Sigma}(\Delta_1)$.  
\end{thm}

\begin{proof} Let $\Sigma_{int}$ be the intersection fan whose cones consist of  intersections of cones $C_i \cap C_j$ where $C_i \in \widehat{\Sigma}(\Delta_1)$ and $C_j \in \Sigma(\Delta_2)$.  The rays of $\widehat{\Sigma}(\Delta_1)$ and $\Sigma(\Delta_2)$ consist only of rays over lattice points in $\Delta_2$.  Because the fans are three-dimensional, any new rays in $\Sigma_{int}$ must be of the form $C_1 \cap C_2$ where $C_2 \in \Sigma(\Delta_2)$ is a two-dimensional cone and $C_1 \in \widehat{\Sigma}(\Delta_1)$ is another two-dimensional cone, and $C_1 \cap C_2$ intersects the relative interior of both $C_i$.  We can now apply Proposition \ref{threedim} with $\Delta = \Delta_2$ to conclude that the primitive integral generator of $C_1 \cap C_2$ must be a lattice point in $\Delta_2$.  

This means that $\Sigma_{int}$ must be a partial crepant subdivision of the fan $\Sigma(\Delta_2)$, since all its rays are rays over lattice points in $\Delta_2$.  It may not be maximal, i.e., it may have rays over only some of the lattice points in $\Delta_2$ as its one-dimensional cones, but using the method of \cite{GKZ}, we may refine $\Sigma_{int}$ to a MPCP resolution $\widehat{\Sigma}(\Delta_2)$.  Since $\Sigma_{int}$ is a subdivision of $\widehat{\Sigma}(\Delta_1)$, $\widehat{\Sigma}(\Delta_2)$ will be as well.
\end{proof}

\section{Four and higher dimensions}

The following generic smoothness lemma will be needed later on.

\begin{lemma} \label{genericsmth} Let $a_0, \dots, a_r \in \mathbb{C}$, and let $$f(z_1, \dots, z_r) = a_0+\sum_{i=1}^r a_i z^{m_i}$$ be a regular function on $$(\mathbb{C}^*)^s \times \mathbb{C}^t \cong \hbox{Spec} \ \mathbb{C}[z^{\pm 1}_1, \dots, z^{\pm 1}_s, z_{s+1}, \dots, z_{s+t}],$$ where $z^{m_i}$ for $1 \leq i \leq r$ are nonconstant Laurent monomials.  Then for generic values of $a_0, \dots, a_r$, the affine variety $V(a_0, \dots, a_r) \subseteq (\mathbb{C}^*)^s \times \mathbb{C}^t$ defined by $f=0$ is nonsingular. \end{lemma}

\begin{proof} Let 
$$Z \subseteq \hbox{Spec} \ \mathbb{C}[a_0, \dots, a_r, z^{\pm 1}_1, \dots, z^{\pm 1}_s, z_{s+1}, \dots, z_{s+t}]$$ 
be the hypersurface defined by $$a_0+\sum_{i=1}^r a_i z^{m_i}=0,$$
and let $$\pi: \hbox{Spec} \ \mathbb{C}[a_0, \dots, a_r, z^{\pm 1}_1, \dots, z^{\pm 1}_s, z_{s+1}, \dots, z_{s+t}] \rightarrow \hbox{Spec} \ \mathbb{C}[a_0, \dots, a_r]$$ be the projection.  The fibers of $\pi|_{Z}$ are the varieties $V(a_0, \dots, a_r)$ for all possible values of $a_0, \dots, a_r$.  The hypersurface $Z$ is nonsingular, because the partial derivative of the defining equation for $Z$ with respect to $a_0$ is 1 and so its gradient is nowhere vanishing.  Applying Sard's theorem, the set of critical values of $\pi|_Z$ has measure zero, and we conclude that the inverse images of $\pi|_Z$ are nonsingular except for a set of values of $a_0, \dots, a_r$ with measure zero.
\end{proof}

In higher dimensions, a result like Theorem \ref{transexistence} is not possible. We will give an example of four dimensional reflexive polytopes $\Delta_1 \subseteq \Delta_2$ where the toric varieties $X(\Delta_1)$ and $X(\Delta_2)$ are already smooth (so no resolution is needed to produce a smooth CY family) but the intersection fan $\Sigma_{int}$ of $\Sigma(\Delta_1)$ and $\Sigma(\Delta_2)$ contains rays not in $\Sigma(\Delta_1)$ or $\Sigma(\Delta_2)$.  This leads to an interesting family of Calabi-Yau threefolds.

For the remainder of the paper, we will let $\Delta_1 \subseteq \Delta_2 \subseteq \mathbb{R}^4$ be the nested four dimensional reflexive polytopes defined by
\begin{gather*}
\Delta_1 = Conv((1,0,0,0),(0,1,0,0),(0,0,1,0), \\
(0,0,0,1),(-1,-1,-1,-1))
\end{gather*}
\begin{gather*}
\Delta_2 = Conv((1,0,0,0),(0,1,0,0),(0,0,1,0),(0,0,0,1), \\ 
(-1,-1,-1,-1),(1,1,1,1),(0,0,0,-1))
\end{gather*}
The toric varieties associated to $\Sigma(\Delta_1)$ and $\Sigma(\Delta_2)$ are smooth, so no MPCP resolutions are required to define smooth Calabi-Yau families $\mathcal{X}(\Delta_1)$ and $\mathcal{X}(\Delta_2)$.  The obvious toric morphism from $X(\Delta_2)$ to $X(\Delta_1)$ fails to exist, because the intersection fan of $\Sigma(\Delta_1)$ and $\Sigma(\Delta_2)$ contains a new ray, the ray over $(1,1,1,0)$.  This ray is the intersection of the cone over $$Conv((1,1,1,1),(0,0,0,-1))$$ in $\Sigma(\Delta_2)$ and the cone over $$Conv((1,0,0,0),(0,1,0,0),(0,0,1,0))$$ in $\Sigma(\Delta_1)$.  The lattice point $(1,1,1,0)$ is not contained in either $\Delta_1$ or $\Delta_2$.

The toric variety $X(\Delta_1)$ is just projective space $\mathbb{P}^4$, and the family $\mathcal{X}(\Delta_1)$ is the familiar family of quintic threefolds.  By setting all coefficients on variables $z^m$ with $m \in (\Delta^*_1 \backslash \Delta^*_2) \cap M$ equal to zero, we obtain a singular family $\mathcal{X}_0(\Delta_1)$.  We will show that generic members of the singular family have two isolated singularities at zero-dimensional toric strata of $\mathbb{P}^4$.  Blowing up at these points results in a smooth family of Calabi-Yau threefolds.

Let $\mathbb{P}^4_{bl}$ be the toric variety obtained by blowing up $\mathbb{P}^4$ at the zero-dimensional toric strata corresponding to the cones over the maximal faces 
$$Conv((1,0,0,0),(0,1,0,0),(0,0,1,0),(0,0,0,1))$$
and
$$Conv((1,0,0,0),(0,1,0,0),(0,0,1,0),(-1,-1,-1,-1)).$$
Thus, the fan for $\mathbb{P}^4_{bl}$, $\Sigma_{bl}$, is a subdivision of the fan $\Sigma(\Delta_1)$ given by adding rays over the points $(1,1,1,1)$ and $(0,0,0,-1)$.  The piecewise linear function $\varphi_{bl} : M_\mathbb{R} \rightarrow \mathbb{R}$ on $\Sigma_{bl}$ corresponding to the anticanonical bundle $\mathcal{L}_{bl}$ of $\mathbb{P}^4_{bl}$ is defined by $\varphi_{bl}(u) = 1$ for each primitive integral generator $u$ of a ray in $\Sigma_{bl}$.  The function $\varphi_{bl}$ is not convex, which can be seen by noting that $$\varphi_{bl}((1,1,1,1)+(0,0,0-1)) = 3 \geq \varphi_{bl}(1,1,1,1)+\varphi_{bl}(0,0,0,-1) = 2.$$  Thus, $\mathcal{L}_{bl}$ is not ample or even semi-ample (generated by its global sections).

The primitive integral generators of rays in $\Sigma_{bl}$ are the same as those in $\Sigma(\Delta_2)$.  We have another piecewise linear function $\varphi_2$ on $\Sigma(\Delta_2)$ which corresponds to the anticanonical bundle $\mathcal{L}_2$ of $X(\Delta_2)$.  Because $\varphi_{bl} \geq \varphi_2$, $\Delta^*_2 = Newt(\varphi_2) \subseteq Newt(\varphi_{bl})$, so all monomial global sections of $\mathcal{L}_2$, $z^m$ with $m \in \Delta^*_2$, can also be regarded as global sections of $\mathcal{L}_{bl}$.  This means that the equation $$\sum_{m \in (\Delta^*_2 \cap M)} c_m z^m = 0,$$ where the $c_m \in \mathbb{C}$ are generic coefficients and the LHS is regarded as a global section of $\mathcal{L}_{bl}$, defines a family $\mathcal{X}_{bl}$ in $\mathbb{P}^4_{bl}$ which is birational to both $\mathcal{X}_0(\Delta_1)$ and $\mathcal{X}(\Delta_2)$.  

\begin{prop} Generic members of the family $\mathcal{X}_{bl}$ are smooth projective Calabi-Yau threefolds and isomorphic to the resolution of generic members of $\mathcal{X}_0(\Delta_1)$ given by blowing up at their two isolated singular points.  \end{prop}

\begin{proof} Lemma \ref{genericsmth} can be applied to show that a generic member of $\mathcal{X}_{bl}$ will be nonsingular everywhere except the zero-dimensional toric strata of $\mathbb{P}^4$ corresponding to the cones over the sets 
$$F_1 = Conv((1,0,0,0),(0,1,0,0),(0,0,1,0),(1,1,1,1))$$
and 
$$F_2 = Conv((1,0,0,0),(0,1,0,0),(0,0,1,0),(0,0,0,-1)).$$
This is because the vertices of $\Delta^*_2$ are the columns of the matrix
\[ \left(
\begin{array}{rrrrrrrrrrrr}
4 & -1 & -1 & 2 & -1 & -1 & 0 & -1 & -1 & 2 & -1 & -1 \\
-1 & 4 & -1 & -1 & 2 & -1 & -1 & 0 & -1 & -1 & 2 & -1 \\
-1 & -1 & 4 & -1 & -1 & 2 & -1 & -1 & 0 & -1 & -1 & 2 \\
-1 & -1 & -1 & -1 & -1 & -1 & 1 & 1 & 1 & 1 & 1 & 1 
\end{array} \right). \]
For every maximal cone $C \in \Sigma_{bl}$, except for the cones over $F_1$ and $F_2$, there is a vertex $m$ in the above list for which $\langle m, n \rangle = -1$ for all primitive integral generators $n$ of extremal rays of $C$.  This means that in local coordinates for the open affine subvariety $U_C \subseteq \mathbb{P}^4_{bl}$ corresponding to $C$, $z^m = 1$, so the local defining equation of $\mathcal{X}_{bl} \cap U_C$ has a leading constant term and the result of Lemma \ref{genericsmth} applies.

For the two other cones $C_1$, $C_2$ which are respectively cones over the sets $F_1$ and $F_2$, the open affine subsets $U_{C_1}$ and $U_{C_2}$ are both isomorphic to $\mathbb{C}^4$.  
The same proof applies to $C_1$ and $C_2$ so we just look at the case of $U_{C_1} \cong \hbox{Spec} \ \mathbb{C}[x,y,z,w]$.  On $U_{C_1}$ the monomials $z^m$ with $m = (0,-1,-1,1)$, $(-1,0,-1,1),$ or $(-1,-1,0,1)$ are free variables (say $x$, $y$, $z$) and all other lattice points in $\Delta^*_2$ correspond to higher order terms, so the defining equation of $\mathcal{X}_{bl} \cap U_{C_1}$ may be written as $f(x,y,z,w) = c_1x+c_2y+c_2z+\cdots=0$ where the omitted monomials are degree 2 or higher in the variables $x, y, z, w$.  If a point of $U_{C_1}$ has not all $x, y, z = 0$, then the point is contained in a chart other than $U_{C_1}$ or $U_{C_2}$ and the argument of the previous paragraph applies.  When $x=y=z=0$, we get that $$(\nabla f)(0,0,0,w) = (c_1 + q_1(w), c_2+q_2(w), c_3+q_3(w), q_4(w)),$$ where $q_1, q_2, q_3, q_4$ are some (not necessarily nonzero) polynomials in $w$ with generic coefficients.  If there is a value of $w$ such that $(\nabla f)(0,0,0,w) = 0$, then we must have that $c_1 = -q_1(w)$, $c_2 = -q_2(w)$, $c_3= -q_3(w)$.  For any particular choice of coefficients of the $q_i$, as $w$ varies over $\mathbb{C}$ this is at most a one-dimensional subset of the space of possible values of $c_1, c_2, c_3 \in \mathbb{C}$.  Thus, generic values of $c_1, c_2, c_3$ will result in a smooth variety.

The fact that each smooth member of $\mathcal{X}_{bl}$ is Calabi-Yau follows from the fact that it is defined by a global section of the anticanonical bundle $\mathcal{L}_{bl}$.  

Lastly, the open subset of a member of $\mathcal{X}_{bl}$ given by removing the exceptional divisors from $\mathbb{P}^4_{bl}$ (corresponding to the rays over $(1,1,1,1)$ and $(0,0,0,-1)$ in $\Sigma_{bl}$) is isomorphic to the open subset of the corresponding member of $\mathcal{X}_0(\Delta_1)$ given by removing the two zero dimensional toric strata corresponding to the cones over
$$Conv((1,0,0,0),(0,1,0,0),(0,0,1,0),(0,0,0,1))$$
and
$$Conv(((1,0,0,0),(0,1,0,0),(0,0,1,0),(-1,-1,-1,-1)).$$
Since we have proven that generic members of $\mathcal{X}_{bl}$ are smooth, this shows that generic members of $\mathcal{X}_0(\Delta_1)$ are smooth except possibly at these two points.  In local coordinates at these points, the homogeneous quintic defining $\mathcal{X}_0(\Delta_1)$ has the form $$p(x_1, x_2, x_3, x_4)+\cdots$$ where $p$ is a homogeneous cubic and the omitted terms are of higher degree.  This results in a singularity where $x_1 = x_2= x_3 =x_4=0$, so members of $\mathcal{X}_0(\Delta_1)$ have singularities at both of these zero dimensional strata.  Blowing up at these points resolves generic members of $\mathcal{X}_0(\Delta_1)$ to smooth members of $\mathcal{X}_{bl}$.
\end{proof}

For the remainder of the paper, we will analyze the topology of members of $\mathcal{X}_{bl}$.  Because this family is birational to the family $\mathcal{X}(\Delta_2)$,  the result of \cite{batyrevbirational} tells us that members of these families have the same Betti numbers.  Thus, to prove they are topologically distinct, we need a finer invariant.  We will approach the problem by calculating the trilinear form induced by cup product on $H^2(Y,\mathbb{Q})$, where $Y$ is the Calabi-Yau manifold of interest.  In the case of both $\mathcal{X}_{bl}$ and $\mathcal{X}(\Delta_2)$, the even cohomology of the Calabi-Yaus is isomorphic to the image of the cohomology of the ambient toric variety.  Thus, for instance, if $i : Y \hookrightarrow \mathbb{P}^4_{bl}$ is inclusion of a member of $\mathcal{X}_{bl}$, the map $i^* : H^i(\mathbb{P}^4_{bl},\mathbb{Q}) \rightarrow H^i(Y,\mathbb{Q})$ is surjective for $i$ even. Furthermore, the kernel of $i^*$ is the same as the kernel of the cup product map $$\cup [Y]: H^i(\mathbb{P}^4_{bl}, \mathbb{Q}) \rightarrow H^{i+2}(\mathbb{P}^4_{bl}, \mathbb{Q}).$$  Since the singular cohomology of a toric variety is well understood, this makes the computations straightforward, especially if aided with computer algebra software.

We have the following theorem which characterizes rational cohomology of complete simplicial toric varieties (\cite{CLS}, Theorem 12.4.1):

\begin{thm}\label{toriccohomo} Let $X(\Sigma)$ be a complete simplicial toric variety and let $\rho_i$ for $1 \leq i \leq r$ be the rays of the fan $\Sigma$.  Let $u_i$ be the primitive integral generator of $\rho_i$ and introduce a variable $x_i$ for each $\rho_i$.  In the ring $\mathbb{Q}[x_1, \dots, x_r]$, let $SR(\Sigma)$ be the Stanley-Reisner ideal 
$$SR(\Sigma) = \langle x_{i_1} \cdots x_{i_s} \ | \ i_j \ \hbox{distinct and} \ \rho_{i_1} + \cdots + \rho_{i_s} \ \hbox{is not a cone of} \ \Sigma \rangle.$$
Let $J \subseteq \mathbb{Q}[x_1, \dots, x_r]$ be the ideal generated by the linear forms $$\sum_{i=1}^r \langle m, u_i \rangle x_i$$ as $m$ ranges over $M$ (equivalently, over a $\mathbb{Z}$-basis of $M$).  Then if each $x_i$ is assigned degree 2, $$\mathbb{Q}[x_1, \dots, x_r]/(J+SR(\Sigma)) \cong H^*(X(\Sigma),\mathbb{Q})$$ as graded rings.  In particular, all odd rational cohomology is zero. \end{thm}

We will also need to use the existence of the Gysin homomorphism and its relation to cup products.  Let $i : Y \hookrightarrow X$ be an inclusion of an irreducible compact complex subvariety $Y$ into a smooth compact complex variety $X$.  Then there is a map $i_! : H^k(Y, \mathbb{Q}) \rightarrow H^{k+2 \dim X - 2 \dim Y}(X, \mathbb{Q})$, such that $i_!(i^*(\alpha)) = \alpha \cup [Y]$ for all $\alpha \in H^*(X, \mathbb{Q})$, where $[Y] \in H^{2 \dim X - 2 \dim Y}(X, \mathbb{Q})$ is the cohomology class of $Y$.  (See \cite{CLS}, Proposition 13.A.6.)

\begin{prop} For smooth members $Y \in \mathcal{X}_{bl}$, the kernel of the map $$i^* : H^*(\mathbb{P}^4_{bl}, \mathbb{Q}) \rightarrow H^*(Y, \mathbb{Q})$$ is equal to the kernel of the cup product map $$\cup [Y] : H^i(\mathbb{P}^4_{bl}, \mathbb{Q}) \rightarrow H^{i+2}(\mathbb{P}^4_{bl}, \mathbb{Q}).$$  The map $i^*$ is surjective onto even-dimensional cohomology of $Y$. \end{prop}

\begin{proof} The fact that the kernel of $i^*$ must be contained in the kernel of $\cup [Y]$ follows automatically from the fact that $\cup [Y] = i_! \circ i^*$.  Thus, we can prove that the two kernels are equal, and that the map is surjective, by showing that $$\dim \ker\ (\cup [Y]:H^i(\mathbb{P}^4_{bl},\mathbb{Q}) \rightarrow H^{i+2}(\mathbb{P}^4_{bl},\mathbb{Q})) = \dim H^i(\mathbb{P}^4_{bl}, \mathbb{Q}) - \dim H^i(Y, \mathbb{Q})$$ for $i = 2, 4, 6$.  The even Betti numbers of $Y$ can be deduced from the result of Batyrev in \cite{batyrevbirational}.  Since $Y$ is birational to a member of the family $\mathcal{X}(\Delta_2)$, it has the same Betti numbers.  These can be calculated with the formulas in \cite{Ba}, yielding that $\dim H^2 (Y, \mathbb{Q}) = 3$.  (Of course, this can also be established by more elementary topological methods without relying on the result of \cite{batyrevbirational}, although the Lefschetz theorem does not directly apply because $Y$ is not an ample divisor of $\mathbb{P}^4_{bl}$.)

The cohomology class $[Y] \in H^2(\mathbb{P}^4_{bl}) \cong \mathbb{Q}[x_1, \dots, x_7]/(J+SR(\Sigma))$ is equal to $x_1+\cdots+x_7$.  This follows from the fact that the divisor $Y$ is linearly equivalent, and therefore cohomologous, to the divisor defined by the vanishing of the line bundle section $s \in \Gamma(\mathcal{L}_{bl},\mathbb{P}^4_{bl})$ corresponding to the lattice point $0 \in Newt(\varphi_{bl})$.  The section $s$ vanishes precisely once on each toric divisor of $\mathbb{P}^4_{bl}$, and the toric divisors are in bijective correspondence with the variables $x_1, \dots, x_7$.  (The same will be true for all other Calabi-Yau hypersurface families in this paper.)

With the description of $H^*(\mathbb{P}^4_{bl}, \mathbb{Q})$ given by Theorem \ref{toriccohomo}, the kernel of $\cup [Y]$ is easily calculated, especially with the aid of computer algebra software such as Macaulay2 \cite{M2}, verifying that it has the expected dimension.  The calculations are summarized in Table \ref{p4blowup}. 
\end{proof}

Quotienting $H^*(\mathbb{P}^4_{bl}, \mathbb{Q})$ by $\ker \cup [Y]$ gives the even cohomology ring of $Y$, which allows the trilinear form on $H^2(Y, \mathbb{Q})$ to be calculated.  The trilinear form is given as a polynomial in $A,B,C$ which represent coefficients of a basis of $H^2(Y,\mathbb{Q})$.  This also depends on a generator of $H^6(Y, \mathbb{Q}) \cong \mathbb{Q}$, so the polynomial is unique up to a $\mathbb{Q}$-linear change of variables and multiplication by a rational number.

Exactly the same process can be carried out for the family $\mathcal{X}(\Delta_2) \subseteq X(\Delta_2)$ which is birational to $\mathcal{X}_{bl}$.  These calculations are summarized in Table \ref{delta2}.  Because members of $\mathcal{X}(\Delta_2)$ are ample divisors of $X(\Delta_2)$, the fact that $H^*(X(\Delta_2), \mathbb{Q})/(\ker \cup [Y])$ is isomorphic to even cohomology of $Y$ was already established by Mavlyutov in \cite{chiralring} (Theorem 5.1).  In fact, the result in \cite{chiralring} only requires that $Y$ be a semi-ample divisor.  Of course, a direct dimension count to verify the statement in this specific case is also possible.

If members of $\mathcal{X}_{bl}$ and $\mathcal{X}(\Delta_2)$ were homeomorphic, then the polynomials $t_{bl}(A,B,C)$ and $t_2(A,B,C)$ would be related by at most an invertible $\mathbb{Q}$-linear transformation of $A,B,C$ and multiplication by a rational number.  To simplify this problem, we instead look at the Milnor rings associated to $t_{bl}$ and $t_2$.  The Milnor ring associated to a rational polynomial $p(A,B,C)$ is defined as $$M(p) = \mathbb{Q}[A,B,C]/D(p)$$ where $D(p)$ is the Jacobian ideal generated by the partial derivatives $p_A$, $p_B$, $p_C$.  If $p$ is homogeneous, then its Milnor ring is naturally graded.  If $p_1$ and $p_2$ are two polynomials related by a linear change of variables in $A$, $B$, $C$, so $$p_1(A,B,C) = p_2(L_1(A,B,C),L_2(A,B,C),L_3(A,B,C))$$ for some rational linear $L_i$, then the mapping from $\mathbb{Q}[A,B,C]$ to itself defined by
{\begin{gather*} A \rightarrow L_1(A,B,C) \\ 
B \rightarrow L_2(A ,B,C) \\ 
C \rightarrow L_3(A,B,C) \end{gather*}}will map $D(p_2)$ to $D(p_1)$ bijectively, as can be seen from the multivariable chain rule.  Thus the Milnor rings $M(p_1)$ and $M(p_2)$ will be graded isomorphic.

Calculation of the Milnor rings $M(t_{bl})$ and $M(t_2)$ shows that they are graded algebras of finite dimension over $\mathbb{Q}$, with graded dimension $(1, 3, 3, 1)$.  We can define a new trilinear form from the Milnor ring by picking a generator $q$ of the degree 3 (highest) part of the algebra, and setting $(aA+bB+cC)^3 = T(a,b,c)q$ where $a, b, c \in \mathbb{Q}$ and $T$ is a homogeneous polynomial of degree 3.  We compute 
\begin{gather*}T_{bl}(a,b,c) = 3abc+3b^2c = 3ab(b+c)\\
T_2(a,b,c) = a^3/5+12a^2b/5+26ab^2/5+4b^3+9a^2c/5- \\16abc/5-2b^2c+ac^2+4bc^2+c^3.
\end{gather*}
According to Macaulay2, the polynomial $T_2(a,b,c)$ is irreducible over $\mathbb{Q}$.  Since $T_{bl}$ is not, this is enough to show that members of the two Calabi-Yau families cannot be homeomorphic.  

The simple form of the polynomial $T_{bl}$ can be attributed to the fact that if $Y \in \mathcal{X}_{bl}$, $H^2(Y, \mathbb{Q})$ has a basis consisting of cohomology classes $[D_i]$ where $D_1$, $D_2$, $D_3$ are mutually non-intersecting divisors on $Y$.  One of the divisors is the intersection of $Y$ with a hyperplane in $\mathbb{P}^4 \backslash \{P_1, P_2 \} \subseteq \mathbb{P}^4_{bl}$ (where $P_1$ and $P_2$ are the blowup points), while the other two are the exceptional divisors that are the inverse images of $P_1$ and $P_2$ in $Y$.  Because of this, the trilinear form on $H^2(Y, \mathbb{Q})$ can be written as $q_1J^3+q_2K^3+q_3L^3$ where $q_i$ are rational numbers and $J$, $K$ and $L$ are the coefficients of $[D_1]$, $[D_2]$ and $[D_3]$ in $H^2(Y, \mathbb{Q})$. 

\subsection{Other hypersurface families}
With a little more work, we can actually show that the family $\mathcal{X}_{bl}$ is topologically distinct from all other smooth Batyrev hypersurface families.  We do this both to show that we have a potentially new topological type of Calabi-Yau threefold, as well as to show the utility of calculating the trilinear form and its Milnor ring in analyzing the topology of Calabi-Yau threefolds.

Because the family $\mathcal{X}_{bl}$ is birational to the family $\mathcal{X}(\Delta_2)$, the Hodge numbers are the same by the result of \cite{batyrevbirational}.  Using the standard formulas in \cite{Ba}, we find that $h^{1,1}=3$ and $h^{1,2} = 79$ for the family $\mathcal{X}(\Delta_2)$. Utilizing the database of Kreuzer and Skarke \cite{ks}, there are a total of four reflexive polytopes associated to smooth Calabi-Yau families with the same Hodge numbers.  One of these is $\Delta_2$.  Data about the other families is summarized in Tables \ref{delta3}-\ref{delta2,5}. Thankfully, for all three of these additional polytopes, the configuration of lattice points is such that the number of possible MPCP subdivisions is low.  $X(\Delta_4)$ is already smooth, while $X(\Delta_3)$ has a unique MPCP resolution $\widehat{X}(\Delta_3)$ obtained by adding the ray over $u_7$, and $X(\Delta_5)$ has two possible MPCP resolutions $\widehat{X}_1(\Delta_5)$ and $\widehat{X}_2(\Delta_5)$ given by choosing different diagonals of the square $Conv(u_1,u_5,u_6,u_7)$.  

For each of these additional Calabi-Yau families we see that the Milnor ring of the trilinear form is either infinite dimensional or of graded dimension $(1,3,3,1)$ (over $\mathbb{Q}$).  For the finite dimensional rings, calculation of the trilinear form on the Milnor ring gives a polynomial which has at most two irreducible factors over $\mathbb{Q}$.  Because the trilinear form on the Milnor ring for $\mathcal{X}_{bl}$ factors completely into linear factors, it must be topologically distinct.

\begin{table}
\begin{tabular}{ |l |m{9cm} | }
  \hline
  Toric variety and CY family & $$\mathcal{X}_{bl} \subseteq \mathbb{P}^4_{bl}$$ \\
  \hline
 Generators of rays of fan & {\begin{gather*} u_1 = (1,0,0,0), \
 u_2 = (0,1,0,0), \\
 u_3 = (0,0,1,0), \
 u_4 = (0,0,0,1), \\
 u_5 = (-1,-1,-1,-1), \
 u_6 = (0,0,0,-1), \\
 u_7 = (1,1,1,1) \end{gather*}} \\
  \hline
  Cohomology ring of $\mathbb{P}^4_{bl}$ & {\begin{gather*} 
  H^*(\mathbb{P}^4_{bl}, \mathbb{Q}) \cong \mathbb{Q}[x_1,\dots,x_7]/(J+SR(\Sigma)) \\
  J = \langle x_1-x_5+x_7, x_2-x_5+x_7, \\
  x_3-x_5+x_7,x_4-x_5-x_6+x_7 \rangle \\
  SR(\Sigma) = \langle x_1x_2x_3x_4,\ x_1x_2x_3x_5,\ x_4x_6,\ x_5x_7,x_6x_7 \rangle \\
  \mathbb{Q}\hbox{-basis } \{1, x_5, x_6, x_7, x^2_5, \\ x^2_6, x^2_7, x^3_5, x^3_6, x^3_7, x^4_7 \}
  \end{gather*}} \\
  \hline
  Kernel of $\cup [Y]$  & 
  $$\langle 3x_5^3-2x_7^3, x_6^3-x_7^3 \rangle \subseteq H^*(\mathbb{P}^4_{bl}, \mathbb{Q})$$ \\
  \hline
  Trilinear form on $H^2(Y,\mathbb{Q})$ & {\begin{gather*} (Ax_5+Bx_6+Cx_7)^3 = \\ (2A^3/3+3A^2B-3AB^2+B^3+C^3)x_7^3 = \\ t_{bl}(A,B,C)x_7^3 \end{gather*}} \\
  \hline  
\end{tabular}
\caption{Data for Calabi-Yau family $\mathcal{X}_{bl} \subseteq \mathbb{P}^4_{bl}$}
\label{p4blowup}
\end{table}

\begin{table}
\begin{tabular}{ |l |m{9cm} | }
  \hline
  Toric variety and CY family & $$\mathcal{X}(\Delta_2) \subseteq X(\Delta_2)$$ \\
  \hline
 Generators of rays of fan & {\begin{gather*} u_1 = (1,0,0,0), \
 u_2 = (0,1,0,0), \\
 u_3 = (0,0,1,0), \
 u_4 = (0,0,0,1), \\
 u_5 = (-1,-1,-1,-1), \
 u_6 = (0,0,0,-1), \\
 u_7 = (1,1,1,1) \end{gather*}} \\
  \hline
  Cohomology ring of $X(\Delta_2)$ & {\begin{gather*} 
  H^*(X(\Delta_2), \mathbb{Q}) \cong \mathbb{Q}[x_1,\dots,x_7]/(J+SR(\Sigma)) \\
  J = \langle x_1-x_5+x_7, x_2-x_5+x_7, \\
  x_3-x_5+x_7,x_4-x_5-x_6+x_7 \rangle \\
  SR(\Sigma) = \langle x_1x_2x_3,\ x_4x_6,\ x_5x_7 \rangle \\
  \mathbb{Q}\hbox{-basis } \{1, x_5, x_6, x_7, x^2_5, \\ x^2_6, x_6x_7, x^2_7, x^3_6, x^2_6x_7, x^3_7, x_6x^3_7 \}
  \end{gather*}} \\
  \hline
  Kernel of $\cup [Y]$  & 
  $$\langle 3x_5^2+2x_6^2+8x_6x_7+5x_7^2 \rangle \subseteq H^*(X(\Delta_2), \mathbb{Q})$$ \\
  \hline
  Trilinear form on $H^2(Y,\mathbb{Q})$ & {\begin{gather*} (Ax_5+Bx_6+Cx_7)^3 = \\ (A^3+9A^2B/2-9AB^2/2+B^3-\\ 3B^2C/2-3BC^2/2+C^3)x_7^3 = \\ t_2(A,B,C)x_7^3 \end{gather*}} \\
  \hline  
\end{tabular}
\caption{Data for Calabi-Yau family $\mathcal{X}(\Delta_2) \subseteq X(\Delta_2)$}
\label{delta2}
\end{table}

%polytopesoutput DELTA3
%{\begin{gather*} \rho_1 = (-1,0,0,0), \
% \rho_2 = (-1,1,0,0), \\
% \rho_3 = (-1,0,1,0), \
% \rho_4 = (1,0,0,-1), \\
% \rho_5 = (2,-1,-1,-1), \
% \rho_6 = (-1,0,0,2), \\
% \rho_7 = (-1,0,0,1) \end{gather*}} \\

%Stanley-Reisner $\langle x_2x_3x_5, x_1x_6, x_4x_7 \rangle$

%Trilinear $t_3(A,B,C) = -2A^2B-2AB^2-2B^3/3-6ABC-3B^2C+6AC^2+3BC^2$

%polytopes3output DELTA4
%{\begin{gather*} \rho_1 = (-1,0,0,0), \
% \rho_2 = (-1,0,-1,0), \\
% \rho_3 = (-1,1,-1,0), \
% \rho_4 = (1,-1,1,-1), \\
% \rho_5 = (1,0,1,0), \
% \rho_6 = (0,0,-1,1), \\
% \rho_7 = (-1,0,-1,1) \end{gather*}} \\
 
%Stanley-Reisner ideal of resolution $\langle x_2x_5, x_1x_6, x_3x_4x_7 \rangle$
 
%Trilinear $t_4(A,B,C) = -5A^3/4-3A^2B-3A^2C/4-6ABC+9AC^2/4-3BC^2/2+C^3$
 
%squarepolytopes output (add line segment from p1 to p6) DELTA5
%{\begin{gather*} \rho_1 = (-1,-1,0,0), \
% \rho_2 = (-1,0,0,-1), \\
% \rho_3 = (0,-1,0,-1), \
% \rho_4 = (-1,-1,0,-1), \\
% \rho_5 = (-1,-1,1,-1), \
% \rho_6 = (-1,0,-1,1), \\
% \rho_7 = (1,1,-1,2) \end{gather*}} \\
 
%Stanley-Reisner $\langle x_1x_2x_3, x_4x_6, x_5x_7 \rangle$

%Trilinear $t_5(A,B,C)=4A^3/5-3A^2C/5+3B^2C$

%squarepolytopes2 output (add line segment from p5 to p7) DELTA5
%{\begin{gather*} \rho_1 = (-1,-1,0,0), \
% \rho_2 = (-1,0,0,-1), \\
% \rho_3 = (0,-1,0,-1), \
% \rho_4 = (-1,-1,0,-1), \\
% \rho_5 = (-1,-1,1,-1), \
% \rho_6 = (-1,0,-1,1), \\
% \rho_7 = (1,1,-1,2) \end{gather*}} \\

%Stanley-Reisner $\langle x_1x_2x_3, x_1x_6, x_4x_6, x_2x_3x_5x_7, x_4x_5x_7 \rangle$

%Trilinear $t_6(A,B,C) = -8A^3/5-B^3+6A^2C/5-3B^2C-3BC^2+C^3$

\begin{table}
\begin{tabular}{ |l |m{9cm} | }
  \hline
  Toric variety and CY family & $$\mathcal{X}(\Delta_3) \subseteq \widehat{X}(\Delta_3)$$ \\
  \hline
 Generators of rays of fan & {\begin{gather*} u_1 = (-1,0,0,0), \
 u_2 = (-1,1,0,0), \\
 u_3 = (-1,0,1,0), \
 u_4 = (1,0,0,-1), \\
 u_5 = (2,-1,-1,-1), \
 u_6 = (-1,0,0,2), \\
 u_7 = (-1,0,0,1) \end{gather*}} \\
  \hline
  Stanley-Reisner ideal of resolution & 
   $$\langle x_2x_3x_5,\ x_1x_6,\ x_4x_7 \rangle$$ \\
  \hline
  Trilinear form & {\begin{gather*} t_3(A,B,C) = -2A^2B-2AB^2-2B^3/3-\\ 6ABC- 3B^2C+6AC^2+3BC^2 \end{gather*}} \\
  \hline  
  Milnor ring & {\begin{gather*} \hbox{Graded dimension} \ (1,3,3,1) \\
  \hbox{Factored trilinear form}\ T_3(a,b,c) = \\ -7a^3/4+6a^2b-39ab^2/5+18b^3/5+\\
  3a^2c-39abc/5+27b^2c/5-\\5ac^2/2+19bc^2/5+c^3 \end{gather*}} \\
  \hline
\end{tabular}
\caption{Data for Calabi-Yau family $\mathcal{X}(\Delta_3) \subseteq \widehat{X}(\Delta_3)$}
\label{delta3}
\end{table}

\begin{table}
\begin{tabular}{ |l |m{9cm} | }
  \hline
  Toric variety and CY family & $$\mathcal{X}(\Delta_4) \subseteq X(\Delta_4)$$ \\
  \hline
 Generators of rays of fan & {\begin{gather*} u_1 = (-1,0,0,0), \
 u_2 = (-1,0,-1,0), \\
 u_3 = (-1,1,-1,0), \
 u_4 = (1,-1,1,-1), \\
 u_5 = (1,0,1,0), \
 u_6 = (0,0,-1,1), \\
 u_7 = (-1,0,-1,1) \end{gather*}} \\
  \hline
  Stanley-Reisner ideal & 
   $$\langle x_2x_5,\ x_1x_6,\ x_3x_4x_7 \rangle$$ \\
  \hline
  Trilinear form & {\begin{gather*} t_4(A,B,C) = -5A^3/4-3A^2B-3A^2C/4- \\ 6ABC+9AC^2/4-3BC^2/2+C^3 \end{gather*}} \\
  \hline  
  Milnor ring & $$\hbox{Infinite dimensional over } \mathbb{Q}$$ \\
  \hline
\end{tabular}
\caption{Data for Calabi-Yau family $\mathcal{X}(\Delta_4) \subseteq X(\Delta_4)$}
\label{delta4}
\end{table}

\begin{table}
\begin{tabular}{ |l |m{9cm} | }
  \hline
  Toric variety and CY family & $$\mathcal{X}_1(\Delta_5) \subseteq \widehat{X}_1(\Delta_5)$$ \\
  \hline
 Generators of rays of fan & {\begin{gather*} u_1 = (-1,-1,0,0), \
 u_2 = (-1,0,0,-1), \\
 u_3 = (0,-1,0,-1), \
 u_4 = (-1,-1,0,-1), \\
 u_5 = (-1,-1,1,-1), \
 u_6 = (1,1,0,1), \\
 u_7 = (1,1,-1,2) \end{gather*}} \\
  \hline
  %diagonal p1 to p6
  Stanley-Reisner ideal of resolution & 
   $$\langle x_1x_2x_3,\ x_4x_6,\ x_5x_7 \rangle$$ \\
  \hline
  Trilinear form & $$t_5(A,B,C)=4A^3/5-3A^2C/5+3B^2C$$ \\
  \hline  
  Milnor ring & $$\hbox{Infinite dimensional over } \mathbb{Q}$$ \\
  \hline
\end{tabular}
\caption{Data for Calabi-Yau family $\mathcal{X}_1(\Delta_5) \subseteq \widehat{X}_1(\Delta_5)$}
\label{delta1,5}
\end{table}

\begin{table}
\begin{tabular}{ |l |m{9cm} | }
  \hline
  Toric variety and CY family & $$\mathcal{X}_2(\Delta_5) \subseteq \widehat{X}_2(\Delta_5)$$ \\
  \hline
 Generators of rays of fan & {\begin{gather*} u_1 = (-1,-1,0,0), \
 u_2 = (-1,0,0,-1), \\
 u_3 = (0,-1,0,-1), \
 u_4 = (-1,-1,0,-1), \\
 u_5 = (-1,-1,1,-1), \
 u_6 = (1,1,0,1), \\
 u_7 = (1,1,-1,2) \end{gather*}} \\
  \hline
  %diagonal p5 to p7
  Stanley-Reisner ideal of resolution & 
   $$\langle x_1x_2x_3,\ x_1x_6,\ x_4x_6,\ x_2x_3x_5x_7,\ x_4x_5x_7 \rangle$$ \\
  \hline
  Trilinear form & {\begin{gather*} t_6(A,B,C) = -8A^3/5-B^3+\\ 6A^2C/5-3B^2C-3BC^2+C^3 \end{gather*}} \\
  \hline  
  Milnor ring & {\begin{gather*} \hbox{Graded dimension} \ (1,3,3,1) \\
  \hbox{Factored trilinear form}\ T_6(a,b,c) = \\ (5a^2-20ab-b^2+20ac+2bc-c^2)b \end{gather*}} \\
  \hline
\end{tabular}
\caption{Data for Calabi-Yau family $\mathcal{X}_2(\Delta_5) \subseteq \widehat{X}_2(\Delta_5)$}
\label{delta2,5}
\end{table}

\bibliography{myrefs}

\begin{thebibliography}{10}

\bibitem{acjm}
A.~C. {Avram}, P.~{Candelas}, D.~{Jancic}, and M.~{Mandelberg}.
\newblock {On the connectedness of the moduli space of Calabi-Yau manifolds}.
\newblock {\em Nuclear Physics B}, 465:458--472, February 1996.

\bibitem{Ba}
V.~V. {Batyrev}.
\newblock {Dual Polyhedra and Mirror Symmetry for Calabi-Yau Hypersurfaces in
  Toric Varieties}.
\newblock {\em {J. Alg. Geom.}}, 3:493--535, 1994.
\newblock arXiv:alg-geom/9310003.

\bibitem{batyrevbirational}
V.~V. {Batyrev}.
\newblock {Birational Calabi--Yau n-folds have equal Betti numbers}.
\newblock October 1997.
\newblock arXiv:alg-geom/9710020.

\bibitem{CLS}
D.~Cox, J.~Little, and H.~Schenk.
\newblock {\em Toric Varieties}.
\newblock American Mathematical Society, 2011.

\bibitem{fred2}
K.~Fredrickson.
\newblock {Resolution of degenerate mirror families via toric morphisms}.
\newblock Preprint, 2013.
\newblock arXiv:1308.0778.

\bibitem{GKZ}
I.M. Gel'fand, A.V. Zelevinskii, and M.M. Kapranov.
\newblock Equations of hypergeometric type and toric varieties.
\newblock {\em Func. Anal. Appl.}, {23}(2):12--26, 1989.
\newblock Engl. trans. 94-106.

\bibitem{M2}
Daniel~R. Grayson and Michael~E. Stillman.
\newblock Macaulay2, a software system for research in algebraic geometry.
\newblock Available at \url{http://www.math.uiuc.edu/Macaulay2/}.

\bibitem{ks}
M.~{Kreuzer} and H.~{Skarke}.
\newblock {Complete classification of reflexive polyhedra in four dimensions}.
\newblock February 2000.
\newblock arXiv:hep-th/0002240.

\bibitem{chiralring}
A.~R. {Mavlyutov}.
\newblock {On the chiral ring of Calabi-Yau hypersurfaces in toric varieties}.
\newblock October 2000.
\newblock arXiv:math/0010318.

\bibitem{morrison}
David~R. Morrison.
\newblock {Through the Looking Glass}.
\newblock In {\em Mirror Symmetry III}, pages 263--277. American Mathematical
  Society and International Press, 1999.

\bibitem{rossi}
Michele Rossi.
\newblock {Geometric Transitions}.
\newblock {\em J. Geom. Phys.}, {56}:1940--1983, 2006.
\newblock arXiv:math/0412514v1.

\bibitem{web}
C.~{Ti-Ming}, B.~R. {Greene}, M.~{Gross}, and Y.~{Kanter}.
\newblock {Black hole condensation and the web of Calabi-Yau manifolds.}
\newblock {\em Nuclear Physics B Proceedings Supplements}, 46:82--95, March
  1996.

\end{thebibliography}
\bibliographystyle{plain}  
\end{document}